\documentclass[12pt]{amsart}
\usepackage{amssymb, amsmath}
\title{Compactness Conditions in Universal Algebraic Geometry}
\author{P. Modabberi \and M. Shahryari}
\address{ P. Modabberi: Department of Pure Mathematics,  Faculty of Mathematical
Sciences, University of Tabriz, Tabriz, Iran}
\email{p\_modabberi@tabrizu.ac.ir}
\address{M. Shahryari: Department of Pure Mathematics,  Faculty of Mathematical
Sciences, University of Tabriz, Tabriz, Iran}
\email{mshahryari@tabrizu.ac.ir}

\newtheorem{example}{Example}
\newtheorem{corollary}{Corollary}
\newtheorem{proposition}{Proposition}

\newtheorem {theorem}{Theorem}
\newtheorem{lemma}{Lemma}

\newtheorem{definition}{Definition}
\begin{document}

\maketitle

\begin{abstract}
In this article, the properties of being equational noetherian, $q_{\omega}$ and $u_{\omega}$-compactness, and equational Artinian
are studied from the perspective of the Zariski topology. The equational conditions on the relative free algebras of arbitrary varieties
are also investigated and their relations to some logic and model theory notions are obtained. Some applications for the case of the universal algebraic geometry over groups are also introduced.
\end{abstract}


\begin{center}{\bf TABLE OF CONTENTS}
\end{center}
\bigskip
\bigskip

1. INTRODUCTION

\bigskip
2. BASIC NOTIONS
\bigskip

\ \ \ 2.1 Systems of equations and algebraic sets

\ \ \ 2.2 Radicals and coordinate algebras

\ \ \ 2.3 Equational noetherian algebras

\ \ \ 2.4 Unification theorems

\bigskip
3. COMPACTNESS CONDITIONS

\bigskip
\ \ \ 3.1 Variants of equational conditions

\ \ \ 3.2 $q_{\omega}$-compactness and Zariski topology

\ \ \ 3.3 Meta-compact algebras

\bigskip
4. RELATIVELY FREE ALGEBRAS

\bigskip
\ \ \ 4.1 Relatively free algebras

\ \ \ 4.2 Finitely axiomatizable classes

\bigskip
5. SOME EQUATIONAL NOETHERIAN GROUPS

\bigskip
6. EQUATIONAL ARTINIAN ALGEBRAS

\bigskip
\ \ \ 6.1 $A$-radicals

\ \ \ 6.2 Relative systems of equations

\ \ \ 6.3 Chain conditions on the ideals of relative free algebras

\ \ \ 6.4 Hilbert's basis theorem

\ \ \ 6.5 Examples of equational Artinian algebras

\bigskip
7. OTHER TYPES OF EQUATIONAL CONDITIONS

\bigskip
7. REFERENCES


\section{Introduction}
Universal algebraic geometry is a new area of modern algebra, whose
subject is basically the study of equations over an arbitrary
algebraic structure $A$. In the classical algebraic geometry $A$ is
a field.  Many articles already published about algebraic geometry
over groups, see \cite{BMR1}, \cite{BMR2}, \cite{BMRom},
\cite{KM}, and \cite{MR}. In an outstanding series of papers, O. Kharlampovich and A. Miyasnikov
developed algebraic geometry over free groups to give
affirmative answer for an old problem of Alfred Tarski concerning
elementary theory of free groups (see \cite{KMTarski} and also \cite{SEL} for the independent solution of Z. Sela). Also in \cite{KMTarski2},
a problem of Tarski about decidablity of the elementary theory of free groups is solved. Algebraic
geometry over algebraic structures is also developed for algebras
other than groups, for example there are results about algebraic
geometry over Lie algebras and monoids, see \cite{DR1}, \cite{MSH},
and \cite{SHEV}. Systematic study of universal algebraic geometry is
done in a series of articles by V. Remeslennikov, A. Myasnikov and
E. Daniyarova in \cite{DMR1}, \cite{DMR2}, \cite{DMR3}, and
\cite{DMR4}.

In this article, we are dealing with the equational conditions in
the universal algebraic geometry, i.e. different conditions relating
systems of equations especially conditions about systems and
sub-systems of equations over algebras. The main examples of such
conditions are equational noetherian property and its variants (weak
equational noetherianity, $n$-equational noetherianity, $q_{\omega}$
and $u_{\omega}$-compactness),
as well as equational Artinian
property of algebras. We begin with a   review of basic concepts of
universal algebraic geometry and  we describe the relation between
the properties of being equational noetherian, $q_{\omega}$ and
$u_{\omega}$-compactness and compactness of certain sets in the
Zariski topology. We also discuss the concept of a meta-compact
algebra and its relation to the meta-compactness of the sets in the
Zariski topology. Then we introduce the notion of equational
Artinian algebras. We provide some necessary conditions
for a relative free algebra to be equational noetherian. We also
show that the equational noetherianity of some relatively free
algebras has interesting logical implications for certain subclasses
of the corresponding variety.  After defining the notion of relative
systems of equations, we show that ascending (descending) chain
conditions on the ideals of the relative free algebra is equivalent
to equational noetherian (equational Artinian) property of all
elements of the corresponding variety. A set of other types of
interesting equational conditions and some questions is also
presented at the end of the article.

\section{Basic notions}
This section is devoted to a fast review of the basic concepts of the universal algebraic geometry. We suggest \cite{BS}, \cite{Gor} and \cite{Mal}
for reader who is not familiar to the universal algebra. The reader also would use \cite{DMR1}, \cite{DMR2}, \cite{DMR3}, and
\cite{DMR4}, for extended exposition of the universal algebraic geometry.  Our notations here are almost the same as in the above mentioned papers.
Many results of this work can be stated for structures over any first order language, but for the sake of simplicity, we restrict ourself for the case
of algebraic languages.

\subsection{ Systems of equations and algebraic sets}
Suppose  $\mathcal{L}$ is an arbitrary algebraic
language and $A$ is a fixed algebra of type $\mathcal{L}$. The extended language will be denoted by $\mathcal{L}(A)$ and it is obtained from $\mathcal{L}$ by adding new constant symbols $a\in A$. An algebra $B$ of type $\mathcal{L}(A)$ is called {\em $A$-algebra}, if the map $a\mapsto a^B$ is an embedding of $A$ in $B$. Note that here, $a^B$ denotes the interpretation of the constant symbol $a$ in $B$.
We assume that $X=\{ x_1, \ldots, x_n\}$ is a finite set of variables. We denote the term algebra in the language $\mathcal{L}$ and variables from $X$ by $T_{\mathcal{L}}(X)$, and similarly the term algebra in the extended language $\mathcal{L}(A)$ will denoted by $T_{\mathcal{L}(A)}(X)$.
For the sake of  simplicity, we define our notions in the coefficient free frame, i.e. in the language $\mathcal{L}$ and then we can extend all the definitions to the language $\mathcal{L}(A)$.

Fix an algebraic language $\mathcal{L}$ and a set of variables $X=\{ x_1, \ldots, x_n\}$. An equation is a pair $(p, q)$ of the elements of the term algebra $T_{\mathcal{L}}(X)$. In many cases, we assume that such an equation is the same as the  atomic formula $p(x_1, \ldots, x_n)\approx q(x_1, \ldots, x_n)$ or $p\approx q$ in short. Hence, in this article the set $At_{\mathcal{L}}(X)$ of atomic formulae in the language $\mathcal{L}$ and the product algebra $T_{\mathcal{L}}(X)\times T_{\mathcal{L}}(X)$ are assumed to be equal.

Any subset $S\subseteq At_{\mathcal{L}}(X)$ is called a {\em system of equations} in the language $\mathcal{L}$. A system $S$ is called {\em consistent} over an algebra $A$, if there is an element $(a_1, \ldots,
a_n)\in A^n$ such that for all equations $(p\approx q)\in S$, the
equality
$$
p^A(a_1, \ldots, a_n)=q^A(a_1, \ldots, a_n)
$$
holds. Otherwise, we say that $S$ is {\em in-consistent} over $A$. Note that, $p^A$ and $q^A$ are the corresponding term functions on $A^n$. A system of equations $S$ is called an ideal,
if it corresponds to a  congruence on $T_{\mathcal{L}}(X)$. For an arbitrary system of
equations $S$, the ideal generated by $S$, is the smallest congruence containing $S$ and it is denoted by $[S]$.

For an algebra $A$ of type $\mathcal{L}$, an element $(a_1, \ldots, a_n)\in A^n$  will be denoted by $\overline{a}$, sometimes. Let  $S$ be  a system of equations. Then the set
$$
V_A(S)=\{ \overline{a}\in A^n: \forall (p\approx q)\in S,\ p^A(\overline{a})=q^A(\overline{a})\}
$$
is called an {\em algebraic set}. It is clear that for any non-empty family $\{ S_i\}_{i\in I}$, we have
$$
V_A(\bigcup_{i\in I}S_i)=\bigcap_{i\in I}V_A(S_i).
$$
So, we define a closed set in $A^n$ to be an arbitrary intersections of finite unions of algebraic sets. Therefore, we obtain a topology on $A^n$,
which is called {\em Zariski topology}. For a subset $Y\subseteq A^n$, its closure with respect to Zariski topology is denoted by $\overline{Y}$.
We also denote by $Y^{ac}$ the smallest algebraic set containing $Y$. In general $\overline{Y}=Y^{ac}$ is not true. We say that an algebra $A$ is
{\em equational domain}, if the union of two algebraic set in $A^n$ is again an algebraic set for any $n$. In this case, clearly we have the equality
$\overline{Y}=Y^{ac}$. In the next section, we will prove some of our results in the case of equational domains, therefore we should mention here that
there is a long  list of known examples of such algebras. For details, the reader can see \cite{DMR4}, where there are also some interesting criteria
for an algebra to be equational domain.

Similarly, we can work with systems of equations with coefficients from a fixed algebra $A$. To do this, let $A$ be an algebra of type $\mathcal{L}$
and consider a pair of terms $(p, q)$ in the extended language $\mathcal{L}(A)$. Then we call the atomic formula $p\approx q$ an equation
with coefficient from $A$. We can define a system of equations with coefficients from $A$, so for any such system $S$ and any $A$-algebra $B$,
we can define the algebraic set $V_B(S)$. Other notions may also be defined in a similar manner. Note that some examples of equational domain appear
in the recent case, for example, as it is proved in \cite{DMR4}, no non-trivial group is equational domain in the language
$\mathcal{L}=(1, ^{-1}, \cdot)$ of groups, but any non-abelian free group $F$ is equational domain in the language $\mathcal{L}(F)$.

Almost all results of the next section may also stated for the
general case of $A$-algebras and systems in the extended language
$\mathcal{L}(A)$, but for the sake of simplicity, we restrict
ourself to the coefficient free case.

\subsection{Radicals and coordinate algebras}
For any set $Y\subseteq A^n$, we define
$$
\mathrm{Rad}(Y)=\{ (p, q): \forall\ \overline{a}\in Y,\ p^A(\overline{a})=q^A(\overline{a})\}.
$$
It is easy to see that $\mathrm{Rad}(Y)$ is an ideal in the term algebra. Any ideal of this type is called an {\em $A$-radical ideal} or a {\em radical ideal} for short.
Note that any ideal in the term algebra is in fact a radical ideal. To see the reason, just note that for any ideal $R$ in the term algebra
$T_{\mathcal{L}}(X)$, if we consider the algebra $B(R)=T_{\mathcal{L}}(X)/R$, then $\mathrm{Rad}_{B(R)}(R)=R$.

It is easy to see that a set $Y$ is algebraic if and only if $V_A(\mathrm{Rad}(Y))=Y$. In the general case, we have $V_A(\mathrm{Rad}(Y))=Y^{ac}$,
see \cite{DMR2}. The {\em coordinate algebra} of a set $Y$ is  the quotient algebra
$$
\Gamma(Y)=\frac{T_{\mathcal{L}}(X)}{\mathrm{Rad}(Y)}.
$$
An arbitrary element of $\Gamma(Y)$ is denoted by $[p]_Y$. We define a function $p^Y:Y\to A$ by the rule
$$
p^Y(\overline{a})=p^A(a_1, \ldots,a_n),
$$
which is a  {\em term function} on $Y$. The  set of all such functions will be denoted by $T(Y)$ and it is naturally an algebra of type $\mathcal{L}$.
It is easy to see that the map $[p]_Y\mapsto p^Y$ is a well-defined isomorphism. So, we have $\Gamma(Y)\cong T(Y)$.

For a system of equation, we can also define the radical $\mathrm{Rad}_A(S)$ to be $\mathrm{Rad}(V_A(S))$. Two systems $S$ and $S^{\prime}$
are called equivalent over $A$, if they have the same set of solutions in $A$, i.e. $V_A(S)=V_A(S^{\prime})$. So, clearly $\mathrm{Rad}_A(S)$ is
the largest system which is equivalent to $S$. Note that $[S]\subseteq \mathrm{Rad}_A(S)$.

One of the major problems of the universal algebraic geometry is to determine the structures of algebras which appear as the coordinate algebras.
There are many necessary and sufficient conditions for an algebra to be a coordinate algebra and we will give a summary of such results in the
subsection 2.4.

\subsection{Equational noetherian algebras}
In this article, we are dealing with equational conditions on algebras. The first and maybe the most important condition of this type can be formulated
as follows.

\begin{definition}
An algebra $A$ is called equational noetherian, if for any system of equations $S$, there exists a finite subsystem $S_0\subseteq S$, which is
equivalent to $S$ over $A$, i.e. $V_A(S)=V_A(S_0)$.
\end{definition}

Note that we can consider a bound on the number of variables appearing in $S$ and obtain a weaker notion of {\em $n$-equational noetherian} algebra;
an algebra $A$ is $n$-equational noetherian, if for any system $S$ with at most $n$ variables, there exists a subsystem $S_0\subseteq S$, which is
equivalent to $S$ over $A$, i.e. $V_A(S)=V_A(S_0)$.

If an $A$-algebra  is equational noetherian in the language $\mathcal{L}(A)$, then we call it $A$-equational noetherian. Many examples of equational
noetherian algebras are introduced in \cite{DMR2}. Among them are noetherian rings and linear groups over noetherian rings as well as  free
groups. In \cite{DMR2}, it is proved that the next four assertions are equivalent:\\

{\em
i- An algebra $A$ is equational noetherian.\\

ii- For any system $S$, there exists a finite $S_0\subseteq [S]$, such that $V_A(S)=V_A(S_0)$.\\

iii- For any $n$, the Zariski topology on $A^n$ is noetherian, i.e. any descending chain of closed subsets terminates.\\

iv- Any chain of coordinate algebras and epimorphisims
$$
\Gamma(Y_1)\to \Gamma(Y_2)\to \Gamma(Y_3)\to \cdots
$$
terminates}.\\

So, in the case of equational noetherian algebras, any closed set in $A^n$ is equal to a minimal finite union of {\em irreducible} algebraic sets which
is unique up to a permutation. Note that a set is called irreducible, if it has no proper finite covering consisting of closed sets.
The following theorem is proved in \cite{DMR2}.

\begin{theorem}
Let $A$ be an equational noetherian algebra. Then the following algebras are also equational noetherian:\\

i- any subalgebra and filter-power of $A$.\\

ii- any coordinate algebra over $A$.\\

iii- any fully residually  $A$-algebra.\\

iv- any algebra belonging to the quasi-variety generated by $A$.\\

v- any algebra universally equivalent to $A$.\\

vi- any limit algebra over $A$.\\

vii- any finitely generated algebra defined by a complete atomic type in the universal theory of $A$ or in the set of quasi-identities of $A$.
\end{theorem}

We can generalize the concept of equational noetherian algebras by dropping the condition $S_0\subseteq S$ in the definition 1. More precisely we have;

\begin{definition}
An algebra $A$ is weak equational noetherian, if for any system $S$ there exists a finite system $S_0$, equivalent to $S$ over $A$.
\end{definition}

Equivalently, an algebra $A$ is weak equational noetherian, if and only if for any system $S$, the radical ideal $\mathrm{Rad}_A(S)$ is finitely
generated, i.e. there exists a finite $S_0\subseteq \mathrm{Rad}_A(S)$ such that $\mathrm{Rad}_A(S)=\mathrm{Rad}_A(S_0)$. We can explain the
logical meaning of this equality as follows. Let $\mathrm{QId}(A)$ be the set of quasi-identities of $A$. Then $A$ is weak equational noetherian iff,
for any system of equations $S$, exist finitely many equation $p_1\approx q_1, \ldots, p_m\approx q_m$ such that
$$
\mathrm{Rad}_A(S)=\{ (p,q):\ (\forall x_1\ldots\forall x_n(\bigwedge_{i=1}^m p_i\approx q_i\Rightarrow p\approx q))\in \mathrm{QId}(A)\}.
$$

\subsection{Unification theorems}
Many variants of the unification theorems are proved for universal algebraic geometry in \cite{DMR1}, \cite{DMR2} and \cite{DMR3}.
The main aim of this type of theorems is to determine  algebras which are the coordinate algebra of an algebraic set. In this subsection, we discuss just one of these unification theorems and the reader can consult the above mentioned articles for detailed exposition of notions and proofs.

\begin{theorem}
Let $A$ and $\Gamma$ be  algebras in a language $\mathcal{L}$. Suppose $A$ is equational noetherian and $\Gamma$ is finitely generated.
Then the following assertions are equivalent.\\

i- $\Gamma$ is the coordinate algebra of some irreducible algebraic set over $A$.\\

ii- $\Gamma$ is a fully residually $A$-algebra. This means that for any finite subset $C\subseteq \Gamma$, there exists a
homomorphism $\alpha:\Gamma\to A$, such that the restriction of $\alpha$ to $C$ is injective.\\

iii- $\Gamma$ embeds into some ultra-power of $A$.\\

iv- $\Gamma$ belongs to the universal closure of $A$, i.e. $Th_{\forall}(A)\subseteq Th_{\forall}(\Gamma)$.\\

v- $\Gamma$ is a limit algebra over $A$.\\

vi- $\Gamma$ is defined by a complete type in $Th_{\forall}(A)$.
\end{theorem}

There are similar theorems for the cases where $A$ is weak equational noetherian, or it is $q_{\omega}$-compact or $u_{\omega}$-compact.
See \cite{DMR3} for a detailed discussion.

\section{Compactness conditions}
In this section we study some of the important equational conditions over algebras by means of the Zariski topology.
The notions of $q_{\omega}$ and $u_{\omega}$-compact algebras are introduced in \cite{DMR3}. In the equational domain case,
we will give a topological characterization of $q_{\omega}$ and $u_{\omega}$-compact algebras. We also introduce a
new class of algebras which generalizes the class of equational noetherian algebras using the concept of {\em meta-compact} topological spaces.

\subsection{Variants of equational conditions}
Recall that an algebra $A$ is equational noetherian if and only if any system of equations is equivalent to a finite subsystem over $A$.
We also defined weak equational noetherian algebras in the previous section. Both of these properties are defined by applying certain conditions on
the systems of equations. Clearly, one can use many different conditions to obtain new classes of algebras. Here we introduce two more examples
of such classes.

\begin{definition}
An algebra $A$ is called $q_{\omega}$-compact if for any system of equations $S$ and any equation $p\approx q$ with $V_A(S)\subseteq V_A(p\approx q)$,
there exists a finite subset $S_0\subseteq S$ such that $V_A(S_0)\subseteq V_A(p\approx q)$. Similarly $A$ is called $u_{\omega}$-compact,
if for any arbitrary system $S$ and any finite system $\{p_1\approx q_1, \ldots, p_m\approx q_m\}$, the inclusion
$$
V_A(S)\subseteq \bigcup_{i=1}^mV_A(p_i\approx q_i)
$$
implies the existence of a finite subset $S_0\subseteq S$ such that
$$
V_A(S_0)\subseteq \bigcup_{i=1}^mV_A(p_i\approx q_i).
$$
\end{definition}

It is easy to see that any equational noetherian algebra is $u_{\omega}$-compact and any $u_{\omega}$-compact algebra is $q_{\omega}$-compact. The converse statements are not true and the reader may see \cite{DMR3} for some counterexamples.  The unification theorem is proved for the algebras in these new classes, \cite{DMR3}. As in the case of $n$-equational
noetherian algebras (2.3), we can also define $q^n_{\omega}$-compact algebras as well as $u^n_{\omega}$-compact algebras.

In the next subsection, we give topological characterization for $q_{\omega}$ and $u_{\omega}$-compact algebras. Note that, an algebra $A$ is
equational noetherian if and only if for all $n$, the space $A^n$ is noetherian and this is equivalent to say that any subset of $A^n$ is
compact in the Zariski topology. To see this, suppose for example every subset of $A^n$ is compact and $S$ is an arbitrary system of equations.
Clearly we have
$$
V_A(S)=\bigcap_{(p\approx q)\in S}V_A(p\approx q),
$$
hence equivalently
$$
A^n\setminus V_A(S)=\bigcup_{(p\approx q)\in S}A^n\setminus V_A(p\approx q).
$$
By the compactness, there are finite number of equations $p_1\approx q_1, \ldots, p_m\approx q_m$ in $S$, such that
$$
A^n\setminus V_A(S)=\bigcup_{i=1}^m A^n\setminus V_A(p_i\approx q_i).
$$
This shows that $V_A(S)=V_A(S_0)$, where $S_0=\{ p_1\approx q_1, \ldots, p_m\approx q_m\}$. Therefore $A$ is equational noetherian.
It can be easily shown that the converse is also true. So we have

\begin{proposition}
An algebra $A$ is equational noetherian if and only if, for any $n$ all subsets of $A^n$ are compact.
\end{proposition}

This proposition is our main motivation to investigate similar criteria for the case of $q_{\omega}$ and $u_{\omega}$-compact algebras.

\subsection{$q_{\omega}$-compactness and Zariski topology}
Let $A$ be an algebra in a language $\mathcal{L}$ and $p\approx q$ be an equation. We denote the open set $A^n\setminus V_A(p\approx q)$
by $C_A(p\approx q)$.

\begin{proposition}
Let $A$ be an equational domain. Then $A$ is $q_{\omega}$-compact if and only if $C_A(p\approx q)$ is compact for all $p\approx q$.
\end{proposition}

\begin{proof}
First suppose $A$ is $q_{\omega}$-compact. Let $C_A(p\approx q)\subseteq \bigcup_{i\in I}C_i$, with $C_i\subseteq A^n$ open. Since $A$ is
equational domain, so $C_i=A^n\setminus V_A(S_i)$ for some system $S_i$. We have
$$
A^n\setminus V_A(p\approx q)\subseteq \bigcup_{i\in I}(A^n\setminus V_A(S_i))=A^n\setminus \bigcap_{i\in I}V_A(S_i),
$$
hence
$$
\bigcap_{i\in I}V_A(S_i)\subseteq V_A(p\approx q).
$$
This shows that
$$
V_A(\bigcup_{i\in I}S_i)\subseteq V_A(p\approx q),
$$
and so, there is a finite  $S^{\prime}\subseteq \bigcup_{i\in I}S_i$, with $V_A(S^{\prime})\subseteq V_A(p\approx q)$. We have
$$
S^{\prime}\subseteq S_{i_1}\cup\cdots \cup S_{i_m},
$$
for some $i_1, \ldots, i_m\in I$. Hence
$$
\bigcap_{j=1}^mV_A(S_{i_j})\subseteq V_A(S^{\prime})\subseteq V_A(p\approx q),
$$
and therefore $C_A(p\approx q)\subseteq \bigcup_{j=1}^m C_{i_j}$. This shows that $C_A(p\approx q)$ is compact. Conversely,
suppose any $C_A(p\approx q)$ is compact. Let $V_A(S)\subseteq V_A(p\approx q)$. Then
\begin{eqnarray*}
C_A(p\approx q)&\subseteq& A^n\setminus V_A(S)\\
               &=& A^n\setminus \bigcap_{(p\approx q)\in S}V_A(p\approx q)\\
               &=&\bigcup_{(p\approx q)\in S}(A^n\setminus V_A(p\approx q)).
\end{eqnarray*}
Hence
$$
C_A(p\approx q)\subseteq \bigcup_{i=1}^m(A^n\setminus V_A(p_i\approx q_i)),
$$
with $p_1\approx q_1, \ldots, p_m\approx q_m\in S$. Therefore
$$
V_A(p_1\approx q_1, \ldots, p_m\approx q_m)\subseteq V_A(p\approx q),
$$
and so, $A$ is $q_{\omega}$-compact.
\end{proof}

A similar result is true for $u_{\omega}$-compact equational domains. It can be shown that an equational domain $A$ is $u_{\omega}$-compact
if and only if any finite intersection of sets of the form $C_A(p\approx q)$ is compact.

\subsection{Meta-compact algebras}
A topological space is called {\em meta-compact} if every open covering of it, has a refinement  which is also a covering and every point belongs
to finitely many element of the refinement. Motivating by meta-compact topological spaces, we define meta-compact algebras. Let $S$ be a system of
equations in a language $\mathcal{L}$ and $A$ be an algebra. We denote by $V_A^{\ast}(S)$ the set of all points $(a_1, \ldots, a_n)\in A^n$ such
that for all but finitely many equations $(p\approx q)\in S$, we have $p^A(a_1, \ldots, a_n)=q^A(a_1, \ldots, a_n)$.

\begin{definition}
Let for any in-consistent  system $S$ over $A$, there exists an in-consistent subsystem $S^{\prime}\subseteq S$, such that $V_A^{\ast}(S^{\prime})=A^n$.
Then we call $A$ a meta-compact algebra.
\end{definition}

Any equational  noetherian algebra is also meta-compact. This is because, if $A$ is equational noetherian and $S$ is an in-consistent system, then
there is a finite $S_o\subseteq S$ such that $V_A(S_0)=V_A(S)$ and so, $S_0$ is also in-consistent. But since $S_0$ is finite, so we have clearly
$V^{\ast}_A(S_0)=A^n$. Hence, $A$ is meta-compact.

\begin{proposition}
Let $A$ be meta-compact equational domain. Then for any $n$ the space $A^n$ is a meta-compact topological space.
\end{proposition}

\begin{proof}
Let $A^n=\bigcup_{\alpha\in I}C_{\alpha}$ be a covering of $A^n$, indexed by a set of ordinals $I=\{ \alpha:\ \alpha\leq \kappa\}$. Since $A$ is
equational domain, every $C_{\alpha}$ has the form
$$
C_{\alpha}=A^n\setminus V_A(S_{\alpha}),
$$
for some system of equations $S_{\alpha}$. We have $\bigcap_{\alpha}V_A(S_{\alpha})=\emptyset$. Suppose $S=\bigcup_{\alpha}S_{\alpha}$.
Then $V_A(S)=\emptyset$, and therefore there exists an in-consistent subsystem $S^{\prime}$, such that $V_A^{\ast}(S^{\prime})=A^n$.
Define by transfinite induction
\begin{eqnarray*}
S_0^{\prime}&=&S^{\prime}\cap S_0,\\
S_{\alpha^{+}}^{\prime}&=&(S^{\prime}\cap S_{\alpha^{+}})\setminus S_{\alpha}^{\prime},
\end{eqnarray*}
and for any limit ordinal, we set
\begin{eqnarray*}
S_{\lambda}^{\prime}&=&(S^{\prime}\cap S_{\lambda})\setminus \bigcup_{\alpha<\lambda}S_{\alpha}^{\prime}.
\end{eqnarray*}
We have clearly $S^{\prime}=\bigcup_{\alpha}S_{\alpha}^{\prime}$, $S_{\alpha}^{\prime}\subseteq S_{\alpha}$, and $S_{\alpha}^{\prime}\cap
S_{\beta}^{\prime}=\emptyset$, for any distinct $\alpha$ and $\beta$. Now, let
$$
C_{\alpha}^{\prime}=A^n\setminus V_A(S_{\alpha}^{\prime}).
$$
We have $C_{\alpha}^{\prime}\subseteq C_{\alpha}$ and $\bigcup_{\alpha}C_{\alpha}^{\prime}=A^n$. Hence, we obtain a refinement of the given covering.
Now, for $\overline{a}\in A^n$, we have $\overline{a}\in V_A^{\ast}(S^{\prime})$, so there are finitely many equations
$$
p_1\approx q_1, \ldots, p_m\approx q_m \in S^{\prime}
$$
such that $p_i^A(\overline{a})\neq q_i^A(\overline{a})$, for $1\leq i\leq m$. For any $i$, there exists a unique $\alpha_i\in I$ such that
$(p_i\approx q_i)\in S_{\alpha_i}^{\prime}$. Therefore, $\overline{a}$ does not belong to
$$
V_A(S_{\alpha_1}^{\prime}), \ldots, V_A(S_{\alpha_m}^{\prime})
$$
and for other $\alpha\in I$, we have $\overline{a}\in V_A(S_{\alpha}^{\prime})$ (since $S_{\alpha_i}^{\prime}\cap S_{\alpha}^{\prime}=\emptyset$).
This shows that $\overline{a}$ just belongs to $C_{\alpha_1}^{\prime}, \ldots, C_{\alpha_m}^{\prime}$. We now, proved that $A^n$ is meta-compact.
\end{proof}

\section{Relatively free algebras}
In this section, we study the universal algebraic geometry of the relatively free algebras. We show that for a variety $\mathbf{V}$, any solution
of an equation over the relatively free algebras of $\mathbf{V}$ is corresponds to an identity in $\mathbf{V}$. This idea, leads us to obtain some
interesting results concerning classes of algebras (especially varieties) using concepts of the universal algebraic geometry.

\subsection{Relatively free algebras}
One of the major tools in the universal algebraic geometry is the {\em relatively free algebra} of a given variety over a given set of variables.
We can discuss this notion in more general framework of {\em pre-varieties}. A class $\mathbf{V}$ of algebras of type $\mathcal{L}$ is a pre-variety,
if it is closed under the operations of taking subalgebra and arbitrary direct product. For an arbitrary algebra $A$, we denote the set of all
congruences of $A$ by $\mathrm{Cong}(A)$. If $\mathbf{V}$ is a pre-variety and $X$ is a  set of variables, we can define an ideal of the term
algebra $T_{\mathcal{L}}(X)$ by
$$
R_{\mathbf{V}}(X)=\bigcap\{ R\in \mathrm{Cong}(T_{\mathcal{L}}(X)): \ \frac{T_{\mathcal{L}}(X)}{R}\in \mathbf{V}\}.
$$
So, $R_{\mathbf{V}}(X)$ is the smallest congruence in the term algebra such that the corresponding quotient belongs to $\mathbf{V}$.
The quotient algebra $$
F_{\mathbf{V}}(X)=\frac{T_{\mathcal{L}}(X)}{R_{\mathbf{V}}(X)}
$$
is called the relative free algebra over $X$ in $\mathbf{V}$. It is a member of $\mathbf{V}$ and it can be characterized by the universal
mapping property: it is generated by the set $\overline{X}=\{ x/R_{\mathbf{V}}(X): x\in X\}$ and any map from $\overline{X}$ to an algebra
$A\in \mathbf{V}$ extends uniquely to a homomorphism from $F_{\mathbf{V}}(X)$ to $A$. It can be easily seen that $|X|=|\overline{X}|$.
More details on the universal mapping property can be find in \cite{BS}. We also have a logical characterization of $R_{\mathbf{V}}(X)$.

\begin{lemma}
Let $\mathbf{V}$ be a pre-variety and $X$ be a set of variables. Then
$$
R_{\mathbf{V}}(X)=\{ (p, q):\ V\vDash (\forall x_1\ldots \forall x_n p\approx q)\}.
$$
\end{lemma}

\begin{proof}
We prove the assertion for finite $X$ and by a small modification, it can be proved for infinite set of variables. Note that if $p$ and $q$
are terms with variables $x_1, \ldots, x_n$, then $\mathbf{V}\vDash (\forall x_1\ldots \forall x_n p\approx q)$ means that for all $A\in \mathbf{V}$
and all $a_1, \ldots, a_n\in A$, we have the equality
$$
p^A(a_1, \ldots, a_n)=q^A(a_1, \ldots, a_n).
$$
To prove the lemma, assume that
$$
K=\{ (p, q):\ V\vDash (\forall x_1\ldots \forall x_n p\approx q)\}.
$$
Let $R$ be an ideal in $T_{\mathcal{L}}(X)$ such that $T_{\mathcal{L}}(X)/R\in \mathbf{V}$ and let $(p,q)\in K$. If we let $A=T_{\mathcal{L}}(X)/R$,
then
$$
p^A(x_1/R, \ldots, x_n/R)=q^A(x_1/R, \ldots, x_n/R),
$$
where $x/R$ denotes the class containing $x$. This equality is equivalent to $p/R=q/R$, so $(p,q)\in R$. This proves that $K\subseteq R_\mathbf{V}$.

To see the inverse inclusion, let $F=T_{\mathcal{L}}(X)/K$, which is generated by the set $X^{\ast}=\{x/K:\ x\in X\}$. We show that $F$ has the
universal mapping property with respect to the set $X^{\ast}$ and the pre-variety $\mathbf{V}$. Let $\alpha: X^{\ast}\to A$ be any map,
where $A\in \mathbf{V}$. Define $\alpha_0:X\to A$ by $\alpha_0(x)=\alpha(x/K)$. We know that there exists a homomorphism
$\alpha_0^{\prime}:T_{\mathcal{L}}(X)\to A$, extending $\alpha_0$. It is easy to see that for all term $p\in T_{\mathcal{L}}(X)$, we have
$\alpha_0^{\prime}(p)=p^A(\alpha(x_1/K), \ldots, \alpha(x_n/K))$. This shows that for $(p,q)\in K$, we have $\alpha_0^{\prime}(p)=\alpha_0^{\prime}(q)$,
and hence $(p,q)\in \ker \alpha_0^{\prime}$. Therefore we have a homomorphism $\alpha^{\prime}:F\to A$ such that
$$
\alpha^{\prime}(t/K)=\alpha_0^{\prime}(t).
$$
Clearly, $\alpha^{\prime}$ coincides with $\alpha$ over $X^{\ast}$. We show that $\alpha^{\prime}$ is unique. Let $h:F\to A$ be another
homomorphism such that $h$ coincides with $\alpha$ over $X^{\ast}$. Using induction on the complexity of the term $t=f(t_1, \ldots, t_m)$, we have
\begin{eqnarray*}
h(t/K)&=& h(\frac{f(t_1, \ldots, t_m)}{K})\\
      &=&f^A(h(t_1/K), \ldots, h(t_m/K))\\
      &=&f^A(\alpha^{\prime}(t_1/K), \ldots, \alpha^{\prime}(t_m/K))\\
      &=&\alpha^{\prime}( \frac{f(t_1, \ldots, t_m)}{K})\\
      &=&\alpha^{\prime}(t/K).
\end{eqnarray*}
This argument shows that $F$ is free relative to $\mathbf{V}$ and hence it belongs to $\mathbf{V}$. Therefore $R_{\mathbf{V}}(X)\subseteq K$.
\end{proof}

We give another interpretation of this lemma using the terminology of the equational logic of Tarski. Recall that a congruence of the term algebra
$T_{\mathcal{L}}(x_1, x_2, \ldots)$ is called {\em fully invariant}, if it is invariant under any endomorphism of the term algebra.
For any set $\Sigma$ of identities, $\Theta_{fi}(\Sigma)$ denotes the fully invariant closure of $\Sigma$. As in \cite{BS}, this set is equal to
the {\em deductive closure} of $\Sigma$, i.e.
$$
\Theta_{fi}(\Sigma)=D(\Sigma).
$$
Now, let $\mathbf{V}$ be a variety of algebras in the language $\mathcal{L}$. Let $\Sigma$ be a set of identities for
$\mathbf{V}$ with variables from $X$. Then the above lemmas says that $R_{\mathbf{V}}(X)=D(\Sigma)$. The next result will be used  in the
subsequence parts of this article.

\begin{corollary}
Let $\mathbf{V}$ be a pre-variety of algebras in a language $\mathcal{L}$ and $X$ be a set. Let $F=F_{\mathbf{V}}(X)$ and $p\approx q$ be an equation
with $n$ indeterminate. Then $(\overline{t}_1, \ldots, \overline{t}_n)\in F^n$ is a solution of $p\approx q$, if and only if
$$
\forall x_1\ldots \forall x_m\ p(t_1, \ldots, t_n)\approx q(t_1, \ldots, t_n)
$$
is an identity in $\mathbf{V}$. Here $x_1, \ldots, x_m$ are the variables appearing in the terms $t_1, \ldots, t_n$.
\end{corollary}

\begin{proof}
Suppose $(\overline{t}_1, \ldots, \overline{t}_n)\in F^n$ is a solution of $p\approx q$. Then we have
$$
p^F(\overline{t}_1, \ldots, \overline{t}_n)=q^F(\overline{t}_1, \ldots, \overline{t}_n),
$$
and therefore
$$
\frac{p(t_1, \ldots, t_n)}{R_{\mathbf{V}}(X)}=\frac{q(t_1, \ldots, t_n)}{R_{\mathbf{V}}(X)}.
$$
This shows that $(p(t_1, \ldots, t_n), q(t_1, \ldots, t_n))\in R_{\mathbf{V}}(X)$ and hence by the above lemma
$$
\mathbf{V}\vDash \forall x_1\ldots\forall x_m  p(t_1, \ldots, t_n)\approx q(t_1, \ldots, t_n).
$$
The converse statement can be proved similarly.
\end{proof}

\begin{corollary}
Let $\mathbf{V}$ be a variety, $X=\{ x_1, \ldots, x_n\}$ and $F=F_{\mathbf{V}}(X)$. Then
$$
\mathrm{Rad}(\overline{x}_1, \ldots, \overline{x}_n)=\mathrm{Id}_X(\mathbf{V}),
$$
where $\mathrm{Id}_X(\mathbf{V})$ denotes the set of identities of $\mathbf{V}$ with $X$ as the set of variables.
\end{corollary}

\subsection{Finitely axiomatizable classes}
We can apply the property of being equational noetherian for certain relatively free algebras, to obtain finite bases of axioms
(consisting of identities) for some classes of algebras. For definition of Horn class, see \cite{BS}. In the next theorem, $\mathrm{Horn}(A)$ is used for the Horn class generated by $A$, i.e. the class of all models of Horn theory of $A$. Note that it is well-known that this class is equal to $P_f(A)$, where $P_f$ denotes the filter product. 

\begin{theorem}
Let $A$ be an algebra of type $\mathcal{L}$ and $\mathbf{V}=\mathrm{Var}(A)$. Let $F_{\mathbf{V}}(X)$ be equational noetherian for all finite $X$.
Suppose $\mathbf{W}$ is a subclass of $\mathrm{Horn}(A)$ axiomatized by  a set of identities $\Sigma\subseteq At_{\mathcal{L}}(x_1, \ldots, x_n)$ with inside the class $\mathrm{Horn}(A)$.
Then there exists a finite subset $\Sigma_0\subseteq \Sigma$ which axiomatizes $\mathbf{W}$ inside $\mathrm{Horn}(A)$, i.e.
$$
\mathbf{W}=\{ B\in \mathrm{Horn}(A):\ B\vDash \Sigma_0\}.
$$
\end{theorem}

\begin{proof}
Suppose $\Sigma=\{ p_i\approx q_i:\ i\in I\}$, so we have
$$
\mathbf{W}=\{ B\in \mathrm{Horn}(A):\ B\vDash \bigwedge_{i\in I}\forall x_1\ldots\forall x_n p_i\approx q_i\}.
$$
Let $X=\{ x_1, \ldots, x_n\}$ and $F=F_{\mathbf{V}}(X)$. We can consider $\Sigma$ as a system of equations over $F$ and since $F$ is equational
noetherian, so there exists a finite $\Sigma_0\subseteq \Sigma$ such that $V_F(\Sigma)=V_F(\Sigma_0)$. Let $I_0$ be the corresponding set of indices,
i.e.
$$
\Sigma_0=\{ p_i\approx q_i:\ i\in I_0\}.
$$
Now, using the corollary 1 in 4.1, for any $t_1, \ldots, t_n$, we have
$$
\mathbf{V}\vDash \bigwedge_{i\in I_0}\forall x_1\ldots\forall x_n p_i(t_1, \ldots, t_n)\approx q_i(t_1, \ldots, t_n),
$$
if and only if
$$
\mathbf{V}\vDash \bigwedge_{i\in I}\forall x_1\ldots\forall x_n p_i(t_1, \ldots, t_n)\approx q_i(t_1, \ldots, t_n).
$$
Since $\mathbf{V}=\mathrm{Var}(A)$, so as a special case we have
$$
A\vDash( \bigwedge_{i\in I_0}\forall x_1\ldots\forall x_n p_i\approx q_i\Rightarrow \bigwedge_{i\in I}\forall x_1\ldots\forall x_n p_i\approx q_i).
$$
This shows that for any $j\in I$, the Horn sentence
$$
\bigwedge_{i\in I_0}\forall x_1\ldots\forall x_n p_i\approx q_i\Rightarrow \forall x_1\ldots\forall x_n p_j\approx q_j
$$
belongs to the Horn theory of $A$. Therefore
$$
\mathrm{Th}_{Horn}(A)+(\bigwedge_{i\in I_0}\forall x_1\ldots\forall x_n p_i\approx q_i)\vDash \forall x_1\ldots\forall x_n p_j\approx q_j,
$$
and hence
$$
\mathbf{W}=\{ B\in \mathrm{Horn}(A):\ B\vDash \bigwedge_{i\in I_0}\forall x_1\ldots\forall x_n p_i\approx q_i\},
$$
therefore $\Sigma_0$ is a set of axioms for $\mathbf{W}$ inside $\mathrm{Horn}(A)$.
\end{proof}

In the following example, we use the fact that every variety is generated by any of its  infinitely generated relatively free elements.
We also use the fact that the free group of the rank two, contains a free group of infinite rank as a subgroup.

\begin{example}
Let $\mathbf{V}$ be the variety of all groups. Clearly $\mathbf{V}=\mathrm{Var}(F_2)$, where $F_2$ is the free group of rank two.
Let $n\geq 1003$ be an odd number and consider the following set of group identities
$$
\Sigma=\{ [x^{pn},y^{pn}]^n\approx 1:\ p=prime\}.
$$
Let $\mathbf{W}_1$ be the variety of groups axiomatized by $\Sigma$. Then, as Adian proves in \cite{Ad}, $\mathbf{W}_1$ is not finitely based,
i.e. it is impossible to axiomatize it using a finite set of identities. Now, suppose
$$
\mathbf{W}=\{ B\in \mathrm{Horn}(F_2):\ B\vDash \Sigma\}=\mathbf{W}_1\cap \mathrm{Horn}(F_2).
$$
Since for any finite $X$, the free group $F(X)$ is equational noetherian, so by the above theorem $\mathbf{W}$ can be axiomatized by a finite subset
of $\Sigma$. This means that there are prime numbers $p_1, \ldots, p_m$ such that
$$
\mathrm{Th}_{Horn}(F_2)+(\bigwedge_{i=1}^m\forall x\forall y[x^{p_in},y^{p_in}]^n\approx 1)\vdash \Sigma.
$$
Hence, although $\Sigma$ is  independent over $\mathrm{Id}(F_2)$, it is not so over the Horn theory of $F_2$.
\end{example}

\section{Some equational noetherian groups}

This section is marginal and it contains some  results on equational noetherian groups. For the sake of generality, we consider the equations with coefficients from a fixed group $A$.

We assume that $A$ is an arbitrary group. An $A$-group is a pair $(G, \lambda)$, where $G$ is a group and $\lambda: A\to G$ is an embedding. If there is no risk of confusion, we will say that $G$ is an $A$-group, and so it contains a distinguished copy of $A$. Let $\mathcal{L}=(\cdot, ^{-1}, 1)$ be the language of groups and for any $a\in A$ attach
a constant symbol $a$ to $\mathcal{L}$. As usual, we denote the extended language by $\mathcal{L}(A)$ and so every $A$-group $G$ becomes an algebra of type $\mathcal{L}(A)$, if we interpret $a$ as $\lambda(a)$. Note that any congruence of $G$ is in fact a normal subgroup $K$ with the property $A\cap K=1$. Through this section, we will call such a normal subgroup an $A$-ideal. We say that $G$ is noetherian if it has maximal condition on the set of $A$-ideals, i.e. any ascending chain of $A$-ideals terminates.

For a set $X=\{ x_1, \ldots, x_n\}$ the free $A$-group generated by $X$ is the free product $A[X]=A\ast F[X]$, where $F[X]$ is the ordinary free group on $X$. We will assume that the embedding $A\hookrightarrow A[X]$ is the inclusion map. Any subset $S\subseteq A[X]$ corresponds to a  system of equations, and if $w=w(x_1,\ldots, x_n, a_1, \ldots, a_m)\in A[X]$ then the expression $w\approx 1$ is an equation with  coefficient $a_1,\ldots, a_m\in A$. Let $(G, \lambda)$ be an $A$-group. We say that $\overline{g}=(g_1, \ldots, g_n )\in G^n$ is a solution for this equation if
$$
w(g_1,\ldots, g_n, \lambda(a_1), \ldots, \lambda(a_m))=1.
$$
For convenience, we will write the above equality as
$$
w(g_1,\ldots, g_n, a_1, \ldots, a_m)=1.
$$
As usual, we denote by $V_G(S)$ the algebraic set corresponding to $S$.

\begin{theorem}
Assume that  $\mathbf{V}$ is variety of $A$-groups. Then all elements of $\mathbf{V}$ are equational noetherian if and only if for all finite $X$, the relatively free $A$-group $F_{\mathbf{V}}(X)$ is noetherian.
\end{theorem}

\begin{proof}
An arbitrary element $wR_{\mathbf{V}}(X)$ in $F_{\mathbf{V}}(X)$ will be denoted by $\overline{w}$. Let $H$ be an element of $\mathbf{V}$ and $h_1, \ldots, h_n\in H$. We define a homomorphism $\varphi:F_{\mathbf{V}}(X)\to H$ by
$$
\varphi(\overline{w})=w(h_1, \ldots, h_n).
$$
Note that this is actually a well-defined map, indeed if $\overline{w}_1=\overline{w}_2$, then $w_1^{-1}w_2\in  R_{\mathbf{V}}(X)$ and so $w_1^{-1}w_2\approx 1$ is an identity in $\mathbf{V}$. Hence, we have
$$
w_1(h_1, \ldots, h_n)=w_2(h_1, \ldots, h_n).
$$
We say that $(h_1, \ldots, h_n)$ is a solution of $\overline{w}\approx 1$ if $\varphi(\overline{w})=1$. For a subset $S\subseteq F_{\mathbf{V}}(X)$, we define
$$
V_H(S)=\{ (h_1, \ldots, h_n)\in H^n:\ \forall \overline{w}\in S\ \  w(h_1, \ldots, h_n)=1\}.
$$
Note that although $S$ is not an ordinary system of equations here, $V_H(S)$ is an ordinary algebraic set in $H^n$.

Now we prove the theorem. The noetherianity of $F_{\mathbf{V}}(X)$ means the $\mathrm{max}$ property for $A$-ideals. So, first let for any finite $X$, the group $F_{\mathbf{V}}(X)$ has $\mathrm{max}$ on $A$-ideals. Let $H\in \mathbf{V}$ and $S\subseteq A[X]$. Let $R$ be the normal closure of $S$ in $A[X]$. Clearly $V_H(S)=V_H(R)$, indeed every element of $R$ has the form $\prod_{i=1}^Nu_iw_i^{\pm 1}u_i^{-1}$, where $w_i\in S$ and $u_i\in A[X]$. We will prove that there is a finite subset $R_0\subseteq R$ such that $V_H(R)=V_(R_0)$. Assume that we done. Let
$$
R_0=\{ v_1, \ldots, v_k\}.
$$
Then for any $i$, we have
$$
v_i=\prod_{j=1}^{N_i}u_{ij}w_{ij}^{\pm 1}u_{ij}^{-1},
$$
with $u_{ij}\in A[X]$ and $w_{ij}\in S$. Let
$$
S_0=\{ w_{ij}:\ 1\leq i\leq k, \ 1\leq j\leq N_i\}.
$$
Then $S_0\subseteq S$ and $V_H(S)=V_H(S_0)$. Therefore, we are going to prove the existence of $R_0$. Let
$$
\overline{S}=\{ \overline{w}\in F_{\mathbf{V}}(X): w\in S\}.
$$
Assume that $\overline{R}$ is the normal closure of $\overline{S}$ in $F_{\mathbf{V}}(X)$. Note that we have
$$
\overline{R}=\{ \overline{w}:\ w\in R\}.
$$
On the other hand we have
$$
V_H(S)=V_H(R)=V_H(\overline{R})=V_H(\overline{S}).
$$
Now, there are two cases:\\

i) $A\cap \overline{R}\neq 1$. Then there exists an element $1\neq a\in A\cap \overline{R}$. Note that $\overline{a}\neq 1$ and so $V_H(\overline{R})=\emptyset$. Hence we can put $R_0=\{ a\}\subseteq R$ in this case.\\

ii) $A\cap \overline{R}=1$. In this case $\overline{R}$ is an $A$-ideal of $F_{\mathbf{V}}(X)$ and so it is finitely generated as an $A$-ideal. Hence there exists a finite $\overline{R}_0\subseteq \overline{R}$, generating $\overline{R}$. We have $V_H(\overline{R})=V_H(\overline{R}_0)$. Suppose $R_0$ is a set of pre-images of elements of $\overline{R}_0$. Then $V_H(\overline{R}_0)=V_H(R_0)$ and so $V_H(R)=V_H(R_0)$. This proves that $H$ is equational noetherian.

Now assume that every element of $\mathbf{V}$ is equational noetherian. Let $X=\{ x_1, \ldots, x_n\}$ be a finite set. We prove that $F_{\mathbf{V}}(X)$ is noetherian. Assume that $K$ is an arbitrary $A$-ideal in $F_{\mathbf{V}}(X)$ and $H=F_{\mathbf{V}}(X)/K$. Then $H$ is an $A$-group belonging to $\mathbf{V}$. Let $\overline{w}\in K$. Then we have
\begin{eqnarray*}
w(\overline{x}_1K, \ldots, \overline{x}_nK)&=& w(\overline{x}_1, \ldots, \overline{x}_n)K\\
                                           &=&\overline{w}K\\
                                           &=&K.
\end{eqnarray*}
This shows that the point $(\overline{x}_1K, \ldots, \overline{x}_nK)\in H^n$ belongs to $V_H(\overline{w}\approx 1)$. Conversely, if $\overline{w}\in F_{\mathbf{V}}(X)$ and $(\overline{x}_1K, \ldots, \overline{x}_nK)$ is a solution of $\overline{w}\approx 1$, then $\overline{w}\in K$. We conclude that
$$
\overline{w}\in K \Leftrightarrow (\overline{x}_1K, \ldots, \overline{x}_nK)\in V_H(\overline{w}\approx 1). \ \ (\ast)
$$
Now, assume that $K_1\varsubsetneq K_2\varsubsetneq K_3\varsubsetneq\cdots $ is a proper chain of $A$-ideals in $F_{\mathbf{V}}(X)$. For any $i$, let $\overline{w}_i\in K_{i+1}\setminus K_i$ and let $L_i$ be the normal closure of the set $K_i\cup \{ \overline{w}_i\}$. We have $K_i\varsubsetneq L_i\subseteq K_{i+1}$. Let $H_i=F_{\mathbf{V}}(X)/K_i$ and $H=\prod_{i=1}^{\infty}H_i$. By assumption $H$ is equational noetherian. For any $i$, we have $H_i\leq H$ and so
$$
(\overline{x}_1K_i, \ldots, \overline{x}_nK_i)\in V_{H_i}(K_i)\subseteq V_H(K_i).
$$
On the other hand $(\overline{x}_1K_i, \ldots, \overline{x}_nK_i)$ does not belong to $V_H(L_i)$, since otherwise $(\overline{x}_1K_i, \ldots, \overline{x}_nK_i)\in V_{H_i}(\overline{w}_i\approx 1)$ which implies by $(\ast)$ that $\overline{w}_i\in K_i$. Hence we have
$$
V_H(K_1)\supsetneq V_H(L_1)\supseteq V_H(K_2)\supsetneq V_H(L_2)\supseteq \cdots
$$ and hence we have the following proper chain of algebraic sets in $H^n$,
$$
V_H(K_1)\supsetneq V_H(K_2)\supsetneq V_H(K_3)\supsetneq \cdots
$$
which is a contradiction.
\end{proof}

Recall that an arbitrary group $A$ is equational noetherian if and only if it is equational noetherian as an $A$-group. A well-known theorem of P. Hall says that any finitely generated  metabelian group $A$ has $\mathrm{max}-n$ property, so we can use the above theorem to prove the next result.\\

\begin{corollary}
Every finitely generated metabelian group is equational noetherian.
\end{corollary}

\begin{proof}
Let $A$ be a finitely generated metabelian group and $\mathbf{V}=Var_A(A)$ be the variety of $A$-groups generated by $A$. Since $A$ satisfies the identity $[[x_1,x_2],[x_3,x_4]]\approx 1$ and this is also an identity in the language $\mathcal{L}(A)$, so every element of $\mathbf{V}$ is metabelian. Now, for a finite set $X$, the group $F_{\mathbf{V}}(X)$ is metabelian and finitely generated since it is generated by $A\cup X$ as an ordinary group. So by  a well-known theorem of P. Hall, it has $\mathrm{max}-n$ property. Therefore, it has also $\mathrm{max}$ on $A$-ideals. So, by the above theorem, every element of $\mathbf{V}$, and specially $A$ is equational noetherian.
\end{proof}

The above result is not new, indeed it is known already by the authors of \cite{BMR1}. Using Hilbert's basis theorem, it is proved that every linear group over a noetherian ring is equational noetherian (see \cite{BMR1}). On the other hand, a theorem of Remeslennikov, says that a finitely generated metabelian group has a faithful representation over a ring which is a direct product of finitely many fields. This shows that every finitely generated metabelian group is equational noetherian. Our method is different in some features: we don't need the result of Remeslennikov, we use Theorem 4 and the above mentioned theorem of Hall. However, our proof applies again Hilbert's basis theorem, because the proof of Hall's theorem depends on it.\\

We obtain a second corollary of the theorem 4, which relates   equational noetherian $A$-group and  finitely based varieties. Note that a variety is {\em finitely based}, if it can be defined by a finite set of identities. A variety has finite {\em axiomatic rank}, if it can be defined by a finite number of variables.

\begin{corollary}
Let $\mathbf{V}$ be a variety of $A$-groups which has finite axiomatic rank. If all elements of $\mathbf{V}$ are equational noetherian, then $\mathbf{V}$ is finitely based.
\end{corollary}

\begin{proof}
Let $X=\{x_1, \ldots, x_n\}$ be the set of  variables which we need to define $\mathbf{V}$ and $R=\mathrm{Id}_X(\mathbf{V})$. It is easy to see that
$$
R=\mathrm{Rad}(\overline{x}_1, \ldots, \overline{x}_n),
$$
and so it is finitely generated as an $A$-ideal of $F_{\mathbf{V}}(X)$ (as we assumed that all elements of $\mathbf{V}$ are equational noetherian). This proves that $\mathbf{V}$ is finitely based.
\end{proof}

With some minor changes, Theorem 4 will be given also for general types of algebras in the section 6 (see subsections 6.2 and 6.3).\\

In the next theorem, we will concentrate on $A$ as an $A$-group. We give a sufficient condition under which a  group $A$ is equational noetherian. An $A$-group $G$ is called {\em finitely cogenerated}, if for any family $\{ K_i\}_{i\in I}$ of $A$-ideals, the assumption $\bigcap_{i\in I}K_i=1$ implies that there is a finite subset $I_0\subseteq I$ such that $\bigcap_{i\in I_0}K_i=1$.

\begin{theorem}
Let $A$ be a group and $\mathbf{V}=\mathrm{Var}_A(A)$ be the variety generated by $A$ as an algebra of type $\mathcal{L}(A)$. Assume that for all $m\geq 1$, all finitely generated subgroup of $A^m$ have $\mathrm{max}-n$. Assume also for all finite $X$, the group $F_{\mathbf{V}}(X)$ is finitely cogenerated. Then $A$ is equational noetherian.
\end{theorem}

\begin{proof}
Recall that $\mathbf{V}=\mathrm{Var}_A(A)$ is the variety of $A$-groups generated by $A$. For any set $X$, we have
$$
R_{\mathbf{V}}(X)=\bigcap\{ R\unlhd A[X]: A\cap R=1,\ \frac{A[X]}{R}\hookrightarrow A\},
$$
and
$$
F_{\mathbf{V}}(X)=\frac{A[X]}{R_{\mathbf{V}}(X)}.
$$

Now assume that
$$
\mathcal{C}=\{ R\unlhd A[X]: A\cap R=1,\ \frac{A[X]}{R}\hookrightarrow A\}.
$$
Let $R\in \mathcal{C}$. Then we have $R/R_{\mathbf{V}}(X)\unlhd F_{\mathbf{V}}(X)$ and this is an $A$-ideal since $A\cap R=1$. Now we have
$$
\bigcap_{R\in \mathcal{C}}\frac{R}{R_{\mathbf{V}}(X)}=1,
$$
so by the assumption of finitely cogeneratedness of $F_{\mathbf{V}}(X)$, there are finitely many elements $R_1, \ldots, R_m$ in $\mathcal{C}$ such that
$$
R_{\mathbf{V}}(X)=\bigcap_{i=1}^m R_i.
$$
Hence we have
$$
F_{\mathbf{V}}(X)=\frac{A[X]}{\bigcap_{i=1}^mR_i}\hookrightarrow \prod_{i=1}^m\frac{A[X]}{R_i}\hookrightarrow A^m.
$$
By the assumption, $F_{\mathbf{V}}(X)$ has $\mathrm{max}-n$, so it has also $\mathrm{max}$ on $A$-ideals. Hence every element of $\mathbf{V}$, specially $A$ itself, is equational noetherian.
\end{proof}

If we apply this theorem for the case of a locally finite group $A$, then the first assumption on the finitely generated subgroups of $A^m$ will be automatically fulfilled. So we obtain our last result.

\begin{corollary}
Let $A$ be a locally finite group and $\mathbf{V}=\mathrm{Var}_A(A)$ be the variety generated by $A$ as an algebra of type $\mathcal{L}(A)$. Assume also for all finite $X$, the group $F_{\mathbf{V}}(X)$ is finitely cogenerated. Then $A$ is equational noetherian.
\end{corollary}

\begin{proof}
Let $\mathbf{V}=\mathrm{Var}_A(A)$. Since $A$ is locally finite, so for all $m\geq 1$, every subgroup of $A^m$ is finite so it has $\mathrm{max}-n$. Now, by the assumption, for all finite $X$, the group $F_{\mathbf{V}}(X)$ is finitely cogenerated. So we  conclude that $A$ is equational noetherian.
\end{proof}

\section{Equational Artinian algebras}
We say that an algebra $A$ is {\em equational Artinian} if every
ascending chain of algebraic sets over $A$ terminates. In this
section, we investigate some properties of equational Artinian
algebras.

\subsection{ $A$-radicals }

One can ask about the existence of  an equational condition, equivalent to being equational Artinian. In this subsection, we will show that the correct condition is not in terms of equations, but rather it can be formulated in terms of radical ideals. We will prove that  $A$ is equational Artinian, iff for any $n$ and $E\subseteq A^n$, there exists a finite subset $E_0\subseteq E$ such that
$$
\mathrm{Rad}(E)=\mathrm{Rad}(E_0).
$$
Note that this condition is in some sense the dual condition of being equational noetherian. First we recall a definition.

\begin{definition}
Let $A$ be an algebra and $E\subseteq A^n$, for some $n$. Then $\mathrm{Rad}(E)$ is  called an $A$-radical ideal of the term algebra
$T_{\mathcal{L}}(x_1, \ldots, x_n)$.
\end{definition}

\begin{theorem}
For an algebra $A$, the following conditions are equivalent;\\

i-  For any $n$ and $E\subseteq A^n$, there exists a finite subset $E_0\subseteq E$ such that
$$
\mathrm{Rad}(E)=\mathrm{Rad}(E_0).
$$

ii- Every descending chain of $A$-radical ideals terminates.\\

iii- $A$ is equational Artinian.
\end{theorem}

\begin{proof}
We first show that $i\Leftrightarrow ii$. Suppose $A$ satisfies $i$. Let
$$
\mathrm{Rad}(E_1)\supseteq \mathrm{Rad}(E_2)\supseteq \mathrm{Rad}(E_3)\supseteq \cdots
$$
be a descending chain of $A$-radicals, with $E_i\subseteq A^n$. Let
$$
E=\bigcup_{i=1}^{\infty}V_A(\mathrm{Rad}(E_i)).
$$
By $i$, there exists a finite $E_0\subseteq E$ such that $\mathrm{Rad}(E)=\mathrm{Rad}(E_0)$.  Since $E_0$ is finite, so there is $k\geq 1$ with
$$
E_0\subseteq V_A(\mathrm{Rad}(E_k)).
$$
Hence, we have
$$
\mathrm{Rad}(E)=\mathrm{Rad}(E_0)\supseteq \mathrm{Rad}(E_k).
$$
On the other hand
$$
\mathrm{Rad}(E)=\bigcap_{i=1}^{\infty}\mathrm{Rad}(E_i)\subseteq \mathrm{Rad}(E_k),
$$
so the chain terminates. Now, suppose $A$ has the property $ii$. We prove $i$. Let $E\subseteq A^n$. Choose an arbitrary  $c_1\in E$. If we have $\mathrm{Rad}(E)=\mathrm{Rad}(\{ c_1\})$, then we  done. So, let $\mathrm{Rad}(E)\varsubsetneq\mathrm{Rad}(\{ c_1\})$. This shows that there is an equation $p\approx q$ such that
$$
(p,q)\in \mathrm{Rad}(\{c_1\})\setminus \mathrm{Rad}(E).
$$
Hence there is a $c_2\in E$ with $p^A(c_2)\neq q^A(c_2)$. Now, if $\mathrm{Rad}(E)=\mathrm{Rad}(\{ c_1, c_2\})$, then the result follows, otherwise we can continue this argument to obtain a non-terminating descending chain of $A$-radicals. Therefore we proved $i\Leftrightarrow ii$. We prove $ii\Rightarrow iii$. Suppose
$$
Y_1\subseteq Y_2\subseteq Y_3\subseteq \cdots
$$
is a chain of algebraic sets in $A^n$. We have $Y_i=V_A(S_i)$, for some system $S_i$. Now,
$$
\mathrm{Rad}(Y_1)\supseteq \mathrm{Rad}(Y_2) \supseteq\mathrm{Rad}(Y_3)\supseteq \cdots
$$
is a descending chain of $A$-radicals and so it terminates, i.e. there is $k$ such that
$$
\mathrm{Rad}(Y_k)=\mathrm{Rad}(Y_{k+1})=\cdots .
$$
This shows that
$$
V_A(\mathrm{Rad}(Y_k))=V_A(\mathrm{Rad}(Y_{k+1}))=\cdots
$$
and hence
$$
Y_k=Y_{k+1}=\cdots .
$$
Therefore $A$ is equational Artinian. Finally, we prove $iii\Rightarrow ii$. Let
$$
\mathrm{Rad}(E_1)\supseteq \mathrm{Rad}(E_2)\supseteq\mathrm{Rad}(E_3)\supseteq \cdots
$$
be a descending chain of $A$-radicals with $E_i\subseteq A^n$. Then we have the following chain of algebraic sets;
$$
E_1^{ac}\subseteq E_2^{ac}\subseteq E_3^{ac}\subseteq \cdots,
$$
and by $iii$ this chain terminates. So, the chain of $A$-radicals already terminates.
\end{proof}

As we saw in the section 3, an algebra $A$ is equational noetherian if and only if every subset of $A^n$ is compact. We have  a similar statement for the case of equational Artinian domains. A subset of a topological space is {\em contra-compact}, if every covering of it by closed sets has a finite subcover. In literature, this type of subset are usually called {\em strongly S-closed}. However, we prefer to use the term "contra-compact".

\begin{theorem}
An equational domain $A$ is equational Artinian, iff for any $n$, every subset of $A^n$ is contra-compact.
\end{theorem}

\begin{proof}
Let a domain $A$ be equational Artinian. Let $C\subseteq A^n$ and
$$
C\subseteq \bigcup_{i\in I}C_i
$$
be a covering of $C$  by closed subsets. Suppose $i_0\in I$ be arbitrary. If $C\subseteq C_{i_0}$, then we have a finite subcover. Otherwise there is $i_1\in I$ such that $C_{i_1}$ is not contained in $C_{i_0}$ and $C\cap C_{i_1}\neq \emptyset$. If we have $C\subseteq C_{i_0}\cup C_{i_1}$, then we have a finite subcover,  otherwise, repeating this process, we obtain a chain
$$
C_{i_0}\varsubsetneq C_{i_0}\cup C_{i_1}\varsubsetneq C_{i_0}\cup C_{i_1}\cup C_{i_2}\varsubsetneq \cdots.
$$
Since $A$ is a domain, all terms of this chain are algebraic sets, and this violates the assumption of being equational Artinian for $A$. Hence every subset of $A^n$ is contra-compact.

Conversely, let every subset of $A^n$ be contra-compact and $E\subseteq A^n$. Since $E$ is contra-compact, there is a finite $E_0\subseteq E$, which is dense in $E$, i.e. $\overline{E}_0=E$. Since $A$ is domain, so we have $E_0^{ac}=E$. Hence we have
$$
\mathrm{Rad}(E_0)=\mathrm{Rad}(E_0^{ac})=\mathrm{Rad}(E).
$$
By the previous theorem $A$ is equational Artinian.

\end{proof}

There is also another topological characterization of equational Artinian algebras under some additional conditions, which we discuss it in the final section (see  the section 6). As a simple application, we discuss a basic property of equational Artinian groups. Let $G$ be a group  and $B\subseteq G^n$. We say that $B$ is an {\em identity base} of $G$, if for all $w\in F(x_1, \ldots, x_n)$, $w_B=1$ implies $(w\approx 1)\in \mathrm{Id}_G(x_1, \ldots, x_n)$. In other words, if all elements of $B$  satisfy the role $w\approx 1$, then $w\approx 1$ is an identity in $G$.

\begin{corollary}
Every equational Artinian group has a finite identity base.
\end{corollary}

\begin{proof}
We have
$$
\mathrm{Id}_G(x_1, \ldots, x_n)=\mathrm{Rad}(G^n).
$$
Since $G$ is equational Artinian, so there is a finite $B\subseteq G^n$ such that $\mathrm{Rad}(G^n)=\mathrm{Rad}(B)$, and so $B$ is an identity base of $G$.
\end{proof}

\subsection{Relative systems of equations}
Let $\mathbf{V}$ be a variety of algebras. In the next subsection, we will  show that  the relative free algebra $F_{\mathbf{V}}(X)$ has descending
chain condition on its ideals, if and only if  every element of $V$ is equational Artinian.  We will do this in a more general context.

In the sequel, we assume that $A$ is an algebra containing a trivial
subalgebra. Suppose $\mathbf{V}$ is a pre-variety of $A$-algebras.
As before, let $X$ be a finite set of variables. Suppose
$R_{\mathbf{V}}(X)$ is the smallest $A$-congruence with the property
$T_{\mathcal{L}(A)}(X)/R_{\mathbf{V}}(X)\in \mathbf{V}$.
Let
$$
F_{\mathbf{V}}(X)=\frac{T_{\mathcal{L}(A)}(X)}{R_{\mathbf{V}}(X)}.
$$
As before, we denote
an arbitrary element of $F_{\mathbf{V}}(X)$ by $\overline{t}$,
where $t$ is a term in $\mathcal{L}(A)$.

Suppose now, $B\in \mathbf{V}$ and $(b_1, \ldots, b_n)\in B^n$. We know that there exists a homomorphism $\varphi:F_{\mathbf{V}}(X)\to B$ such that
$$
\varphi(\overline{p})=p^B(b_1, \ldots, b_n).
$$
Therefore, if $\overline{p}_1=\overline{p}_2$, then $p_1^B(b_1, \ldots, b_n)=p_2^B(b_1, \ldots, b_n)$. This shows that the following definition
has no ambiguity.

\begin{definition}
A $\mathbf{V}$-equation is an expression of the form $\overline{p}\approx \overline{q}$, where $p$ and $q$ are terms in the language
$\mathcal{L}(A)$. If $B$ is an $A$-algebra and $(b_1, \ldots, b_n)$ is an element of $B^n$, we say that $(b_1, \ldots, b_n)$ is a solution of
$\overline{p}\approx \overline{q}$, if $p^B(b_1, \ldots, b_n)=q^B(b_1, \ldots, b_n)$.
\end{definition}

Let $S$ be a system of $\mathbf{V}$-equations. The set of all solutions of elements of $S$, will be denoted by $V_B^{\mathbf{V}}(S)$.
The following observation shows that this is an ordinary algebraic set. Let $S^{\prime}$ be the set of all equations $p\approx q$ such that
$\overline{p}\approx \overline{q}$ belongs to $S$. Then it can be easily verified that
$$
V_B^{\mathbf{V}}(S)=V_B(S^{\prime}).
$$
Therefore, in the sequel we will denote the algebraic set $V_B^{\mathbf{V}}(S)$ by the same notation $V_B(S)$. The Zariski topology arising from
algebraic sets relative to the pre-variety $\mathbf{V}$ is the same as the ordinary Zariski topology.  If $Y\subseteq B^n$, we define
$$
\mathrm{Rad}_B^{\mathbf{V}}(Y)=\{ \overline{p}\approx \overline{q}: \forall \overline{b}\in Y\ p^B(b_1, \ldots, b_n)=q^B(b_1, \ldots, b_n)\}.
$$
The quotient algebra
$$
\Gamma_{\mathbf{V}}(Y)=\frac{F_{\mathbf{V}}(X)}{\mathrm{Rad}_B^{\mathbf{V}}(Y)}
$$
is the {\em $\mathbf{V}$-coordinate algebra} of $Y$. Again, it is easy to see that $\Gamma_{\mathbf{V}}(Y)\cong \Gamma(Y)$.

\subsection{Chain conditions on the ideals of relatively free algebras}
An algebra $B$ will be called noetherian (Artinian), if any ascending (descending) chain of ideals in $B$ terminates. In the case of $A$-algebras,
we restrict ourself to {\em $A$-ideals}. A congruence $R$ in $B$ is called $A$-ideal, if for all $a_1, a_2\in A$, the assumption $(a_1, a_2)\in R$
implies $a_1=a_2$. The following theorem is is very similar to Theorem 4 of the section 5, and so we don't prove it here (for a proof one can see also \cite{Shah}).

\begin{theorem}
Let $\mathfrak{Y}$ be a variety of algebras of type $\mathcal{L}$ and $A\in \mathfrak{Y}$ containing a trivial subalgebra. Let
$\mathbf{V}=\mathfrak{Y}_A$ be the class of elements of $\mathfrak{Y}$ which are $A$-algebra. Then for any finite $X$, the relatively free algebra
$F_{\mathbf{V}}(X)$ is noetherian if and only if every $B\in \mathbf{V}$ is $A$-equationally noetherian.
\end{theorem}

Note that an $A$-algebra $B$ is called $A$-equational noetherian ($A$-equational Artinian), if it is equational noetherian (equational Artinian)
as an algebra of type $\mathcal{L}(A)$.

We are now ready to prove the analogue of the above theorem for the property of being Artinian.

\begin{theorem}
Let $\mathfrak{Y}$ be a variety of algebras of type $\mathcal{L}$ and $A\in \mathfrak{Y}$ containing a trivial subalgebra. Let
$\mathbf{V}=\mathfrak{Y}_A$ be the class of elements of $\mathfrak{Y}$ which are $A$-algebra. Then for any finite $X$, the relatively free algebra
$F_{\mathbf{V}}(X)$ is  Artinian, if and only if every $B\in \mathbf{V}$ is $A$-equationally Artinian.
\end{theorem}

\begin{proof}
The main idea of the proof is the same as in Theorem 4. First, suppose that $F_{\mathbf{V}}(X)$ is Artinian and $B\in \mathbf{V}$. Let
$$
Y_1\subseteq Y_2\subseteq Y_3\subseteq \cdots
$$
be a chain of algebraic sets in $B^n$. Then we have
$$
\mathrm{Rad}(Y_1)\supseteq \mathrm{Rad}(Y_2)\supseteq \mathrm{Rad}(Y_3)\supseteq \cdots,
$$
which is a chain of $A$-ideals in $F_{\mathbf{V}}(X)$. So, it terminates and hence, there is $m$ such that
$$
\mathrm{Rad}(Y_m)=\mathrm{Rad}(Y_{m+1})=\cdots.
$$
This implies that $Y_m=Y_{m=1}=\cdots$ and therefore $B$ is equational Artinian. Now, suppose $B$ is equational Artinian for all $B\in \mathbf{V}$.
If $R$ is an $A$-ideal in $F_{\mathbf{V}}(X)$, we put $B(R)=F_{\mathbf{V}}(X)/R$, which is belong to $\mathbf{V}$. It is easy to see that
$(\overline{p}, \overline{q})\in R$ if and only if
$$
(\overline{x}_1/R, \ldots, \overline{x}_n/R)\in V_{B(R)}(\overline{p}\approx \overline{q}).
$$
Now, let
$$
R_1\supsetneq R_2\supsetneq R_3\supsetneq\cdots
$$
be a proper descending chain of $A$-ideals in $F_{\mathbf{V}}(X)$. Note that in the same time, $R_i$ is a system of $\mathbf{V}$-equations.
Suppose $(\overline{p}_i, \overline{q}_i)\in R_i\setminus R_{i+1}$. Suppose also that $T_i$ is the $A$-ideal generated by the set
$R_{i+1}+(\overline{p}_i, \overline{q}_i)$. Then we have
$$
R_{i+1}\varsubsetneq T_i\varsubsetneq R_i.
$$
Suppose $B_i=B(R_i)$ and $B=\prod_iB_i$. So $B\in\mathbf{V}$. Hence it is equational Artinian. Note that, since $A$ contains trivial algebra,
so $B_i\leq B$. Hence
$$
U=(\overline{x}_1/R_{i+1}, \ldots, \overline{x}_n/R_{i+1})\in V_{B_{i+1}}(R_{i+1})\subseteq V_B(R_{i+1}).
$$
But, $U$ does not belong to $V_B(T_i)$, since otherwise, $U\in V_B(\overline{p}_i\approx \overline{q}_i)$ which implies that
$(\overline{p}_i, \overline{q}_i)\in R_{i+1}$. Therefore, we have
$$
V_B(R_i)\subseteq V_B(T_i)\varsubsetneq V_B(R_{i+1}),
$$
so the chain $V_B(R_1)\subsetneqq V_B(R_2)\subsetneqq V_B(R_3)\subsetneqq\cdots$ is a proper chain of algebraic sets, which is a contradiction.

\end{proof}

\subsection{Hilbert's basis theorem}
A universal algebraic version of the Hilbert's basis theorem is given in \cite{Shah}. In this subsection, we discuss briefly some results on this item.
Suppose $\mathcal{L}$ is an algebraic language and $\mathfrak{Y}$ is a
variety of algebras of type $\mathcal{L}$. Let $A\in \mathfrak{Y}$ and $\mathbf{V}=\mathfrak{Y}_A$ be the class of all elements of $\mathfrak{Y}$
which are $A$-algebra. If $A$ has maximal property on its ideals, is the algebra $F_{\mathbf{V}}(X)$ noetherian?

\begin{example}
Let $\mathcal{L}=(0, 1, +, \times)$ be the language of unital rings and $\mathfrak{Y}$ be the variety of all commutative rings with unite element.
Let $A\in \mathfrak{Y}$ and  $\mathbf{V}=\mathfrak{Y}_A$. If $X=\{ x_1, \ldots, x_n\}$, then $F_{\mathbf{V}}(X)=A[x_1, \ldots, x_n]$ and hence
Hilbert's basis theorem is valid in this case.
\end{example}

\begin{example}
Let $\mathcal{L}=(e, ^{-1}, \cdot)$ be the language of groups. Let $\mathfrak{Y}$ be the variety of groups. Let $A$ be any group and
$\mathbf{V}=\mathfrak{Y}_A$. Then $F_{\mathbf{V}}(X)=A\ast F(X)$. We show that $F_{\mathbf{V}}(X)$ is not noetherian even if $A$ has
maximal property on its normal subgroups (max-n). Consider the Baumslag-Solitar group
$$
B_{m, n}=\langle a, t: ta^mt^{-1}=a^n\rangle,
$$
where $m, n\geq 1$ and $m\neq n$. Then, as is proved in \cite{BMR1},
this group is not equationally noetherian. Let $B=A\ast B_{m,n}$.
Then $B$ is an $A$-group which is not $A$-equationally noetherian.
So, by the theorem 8 of the previous subsection, $A\ast F(X)$ is not noetherian, Hilbert's
basis theorem fails.
\end{example}

\begin{example}
Let $\mathfrak{Y}$ be the variety of abelian groups and $A\in \mathfrak{Y}$ be finitely generated. Suppose $\mathbf{V}=\mathfrak{Y}_A$.
Then it is easy to see that $F_{\mathbf{V}}(X)=A\times F_{ab}(X)$, where $F_{ab}(X)$ is the free abelian group generated by $X$. So,
$F_{\mathbf{V}}(X)=A\times \mathbb{Z}^n$. As a $\mathbb{Z}$-module, clearly $A\times \mathbb{Z}^n$ is noetherian, so Hilbert's basis theorem is true
for any finitely generated abelian group $A$ in the variety of abelian groups. As a result, every abelian group $B$ containing $A$
is $A$-equationally noetherian.
\end{example}

As we mentioned above, if $A\leq B$ and $B$ is not equationally noetherian, then it is also not $A$-equationally noetherian. So, let
$\mathfrak{Y}$ be a variety of algebras and $A\in \mathfrak{Y}$. Let $\mathbf{V}=\mathfrak{Y}_A$. If there exists an element
$B\in \mathfrak{Y}$ which is not equationally noetherian, then by our theorem, $F_{\mathbf{V}}(X)$ is not noetherian, so we never have a
version of Hilbert's basis theorem for the variety $\mathfrak{Y}$.

\begin{example}
Let $\mathfrak{Y}$ be the variety of nilpotent groups of class at most $c$. If $A\in \mathfrak{Y}$ and $\mathbf{V}=\mathfrak{Y}_A$ and
$B\in \mathfrak{Y}$ is not finitely generated, then by \cite{MR}, $B$ is not equationally noetherian and hence $F_{\mathbf{V}}(X)$ is not noetherian.
\end{example}

\subsection{Examples of equational Artinian algebras}
In this subsection, we give some examples of the equational Artinian
algebras.

\begin{example}
Suppose $A$ is an equational Artinian algebra (for example, a finite
algebra). Then for any set $I$, the algebra $A^I$ is also equational
Artinian. This is true, because for any equation $p\approx q$, there
is a natural bijection between the  sets $V_A(p\approx q)^I$ and
$V_{A^I}(p\approx q)$, given by
$$
(a^1_i, \ldots, a^n_i)_{i\in I}\mapsto ((a^1_i)_{i\in I}, \ldots,
(a^n_i)_{i\in I}).
$$
Hence, any chain
$$
V_{A^I}(S_1)\subseteq V_{A^I}(S_2)\subseteq V_{A^I}(S_3)\subseteq
\cdots
$$
becomes
$$
V_A(S_1)^I\subseteq V_A(S_2)^I\subseteq V_A(S_3)^I\subseteq \cdots,
$$
and consequently, we obtain a chain
$$
V_A(S_1)\subseteq V_A(S_2)\subseteq V_A(S_3)\subseteq \cdots,
$$
which terminates.
\end{example}

\begin{example}
Let $R$ be a noetherian ring and $A=(R, +, -)$. Then $A$ is
equational Artinian. To prove this, let $p=a_1x_1+\cdots+a_nx_n$
with $a_i\in \mathbb{Z}$, be a term. We show that $V_A(p\approx 0)$
is a submodule of $R^n$. Clearly, this algebraic set is closed under
addition, so let $\lambda\in R$. Then the map $\alpha_{\lambda}:R\to
R$, $\alpha_{\lambda}(x)=\lambda x$, is a homomorphism of $A$. Since
every algebraic set is invariant under homomorphisms, so we have
$\lambda V_A(p\approx 0)\subseteq V_A(p\approx 0)$. This shows that
for any system $S$, the algebraic set $V_A(S)$ is a submodule of
$R^n$ ($n$ is the number of indeterminate in $S$). But $R^n$ is
noetherian and hence $A$ is equational Artinian.
\end{example}

\begin{example}
Let $R$ be a noetherian ring and $\Lambda=(\lambda_1, \ldots,
\lambda_m)$ be a tuple of elements of $R$ with $\sum_i\lambda_i=1$.
Define an $m$-ary operation
$$
p_{\Lambda}(x_1, \ldots, x_m)=\sum_i\lambda_ix_i.
$$
Then by a similar argument as in the previous example, we see that
$A=(R, p_{\Lambda})$ is equational Artinian. If we assume that $R$
is noetherian but not Artinian ring, then we obtain example of an
equational Artinian algebra which is not equational noetherian.
\end{example}

\begin{example}
It can be shown that only finite fields are equational Artinian in the language of rings $\mathcal{L}=(+, \cdot, 0, 1)$.
\end{example}

\section{Other types of equational conditions}
During this article, we saw many cases of equational conditions in the universal algebraic geometry. It is easy to find new kinds of equational
conditions.  In this final section, we give two more examples of this kind of conditions.

Let $A$ be an algebra of type $\mathcal{L}$. A system $S$ is called $A$-{\em independent} if for any finite subset $S_0\subseteq S$, there exists an
equation $(p\approx q)\in S_0$, such that $V_A(S_0)\varsubsetneq V_A(S_0\setminus p\approx q)$.

\begin{theorem}
For any algebra $A$ and any system $S$, there exists an $A$-independent subsystem $S^{\prime}$ equivalent to $S$ over $A$.
\end{theorem}

\begin{proof}
Suppose $A$ and $S$ are given and $\mathcal{F}$ is the collection of all $A$-independent subsets of $S$. Using Zorn's lemma, we show that
$\mathcal{F}$ has a maximal element. Let $\{ T_{\alpha}\}_{\alpha}\subseteq \mathcal{F}$ be a chain. Let $T=\bigcup T_{\alpha}$ and $T^0\subseteq T$
be  finite. Then $T^0\subseteq T_{\alpha}$ for some $\alpha$. Hence, there exists $(p\approx q)\in T^0$, with
$$
V_A(T^0)\varsubsetneq V_A(T^0\setminus p\approx q).
$$
This shows that $T$ is $A$-independent and so the chain has upper bound. Therefore $\mathcal{F}$ has a maximal element $S^{\prime}$.
Suppose $V_A(S)\varsubsetneq V_A(S^{\prime})$. So, there exists an element $\overline{a}\in V_A(S^{\prime})\setminus V_A(S)$,
and hence there is an equation $(p_1\approx q_1)\in S$, such that $p_1(\overline{a})\neq q_1(\overline{a})$. Let
$S^{\prime\prime}=S^{\prime}+(p_1\approx q_1)$. We show that $S^{\prime\prime}\in \mathcal{F}$. Let $T_0\subseteq S^{\prime\prime}$ be finite.
If $T_0\subseteq S^{\prime}$, then there is $(p\approx q)\in T_0$ with
$$
V_A(T_0)\varsubsetneq V_A(T_0\setminus p\approx q).
$$
If $(p_1\approx q_1)\in T_0$, then $T_0\setminus p_1\approx q_1\subseteq S^{\prime}$, and hence
$$
\overline{a}\in V_A(T_0\setminus p_1\approx q_1).
$$
But, since $p_1(\overline{a})\neq q_1(\overline{a})$, so $\overline{a}$ does not belong to $V_A(T_0)$. Therefore
$$
V_A(T_0)\varsubsetneq V_A(T_0\setminus p_1\approx q_1),
$$
and hence $S^{\prime\prime}\in \mathcal{F}$, which is impossible. Hence we must have $V_A(S)=V_A(S^{\prime})$.
\end{proof}

A topological space $M$ is $\omega$-cocompact if and only if for any countable open covering $M=\bigcup_{i=1}^{\infty}C_i$,
there exists $m\geq 1$ such that for all $j_1, j_2\geq m$, we have
$$
\bigcup_{i=j_1}^{\infty}C_i=\bigcup_{i=j_2}^{\infty}C_i.
$$

\begin{lemma}
If a topological space $M$ is Artinian, then every subset of $M$ is $\omega$-cocompact. The converse is also true.
\end{lemma}

\begin{proof}
We first show that for an Artinian topological space $M$, every $N\subseteq M$ is also Artinian. Suppose
$$
Y_1^{\prime}\subseteq Y_2^{\prime}\subseteq Y_3^{\prime}\subseteq \cdots
$$
is a chain of closed sets in $N$. We have $Y_i^{\prime}=Y_i\cap N$, for all $i$, where $Y_i$ is closed in $M$. Let $V_i=Y_1\cup\cdots\cup Y_i$.
Then $V_i$ is closed in $M$ and we have
$$
V_1\subseteq V_2\subseteq V_3\subseteq \cdots.
$$
So, there is $m$ such that $V_m=V_{m+1}=V_{m+2}=\cdots$. This shows that
$$
Y_1\cup\cdots \cup Y_m=Y_1\cup\cdots \cup Y_m\cup Y_{m+1}=\cdots,
$$
and hence
$$
Y_1^{\prime}\cup\cdots \cup Y_m^{\prime}=Y_1^{\prime}\cup\cdots \cup Y_m^{\prime}\cup Y_{m+1}^{\prime}=\cdots,
$$
which implies that $Y_m^{\prime}=Y_{m+1}^{\prime}=\cdots$. This shows that $N$ is also Artinian. Now, suppose $M$ is Artinian
and $M=\bigcup_{i=1}^{\infty}C_i$. We have $C_i=M\setminus D_i$, with $D_i$ closed. Then $\bigcap_{i=1}^{\infty}D_i=\emptyset$. Define
$$
Y_m=\bigcap_{i=m}^{\infty}D_i.
$$
Then $Y_1\subseteq Y_2\subseteq Y_3\subseteq \cdots$ is a chain of closed sets and so there exists $m$ with $Y_m=Y_{m+1}=\cdots$.
Hence, for$j_1, j_2\geq m$ we have
$$
\bigcap_{i=j_1}^{\infty}D_i=\bigcap_{i=j_2}^{\infty}D_i,
$$
and so
$$
\bigcup_{i=j_1}^{\infty}C_i=\bigcup_{i=j_2}^{\infty}C_i.
$$
Hence $M$ is $\omega$-cocompact. Conversely, suppose $M$ is $\omega$-cocompact and $
$$
Y_1\subseteq Y_2\subseteq Y_3\subseteq \cdots$ is a chain of closed subsets of $M$. Then we have
$$
M=\bigcup_{i=1}^{\infty}Y_i^c\cup M,
$$
where the superscript $^c$ denotes the complement. So, there exists $m$ such that for all $j_1, j_2\geq m$,
$$
\bigcup_{i=j_1}^{\infty}Y_i^c=\bigcup_{i=j_2}^{\infty}Y_i^c.
$$
This implies that
$$
\bigcap_{i=m}^{\infty}Y_i=\bigcap_{i=m+1}^{\infty}Y_i=\cdots.
$$
Therefore, we have
$$
Y_m=\bigcap_{i=m+1}^{\infty}Y_i=\bigcap_{i=m+2}^{\infty}Y_i.
$$
Similarly, we have
$$
Y_{m+1}=\bigcap_{i=m+2}^{\infty}Y_i=\cdots.
$$
Hence $Y_m=Y_{m+1}=\cdots$, and so $M$ is Artinian.

\end{proof}

By this lemma,  a domain $A$ is equational Artinian, if and only if, for all $n$, any subset of $A^n$ is $\omega$-cocompact.
Let $A$ be an algebra and $S$ be a system of equation. We say that $S$ is {\em $A$-stable}, if for any proper finite subset $S^{\prime}\subseteq S$,
we have $V_A(S)=V_A(S\setminus S^{\prime})$.

\begin{theorem}
Let $A$ be an equational domain which is weak equational noetherian. Suppose for any system $S$, there exists a finite subset $S_0$ such that
$S\setminus S_0$ is $A$-stable. Then $A$ is equational Artinian. The converse is true if in addition, the language $\mathcal{L}$ is countable.
\end{theorem}

\begin{proof}
Suppose for any system $S$, there exists a finite subset $S_0$ such that $S\setminus S_0$ is $A$-stable. We show that $A^n$ is $\omega$-cocompact.
Let $A=\bigcup_{i=1}^{\infty}C_i$ be an open covering. We have $C_i=A^n\setminus V_A(S_i)$ for some finite $S_i$. So,
$\bigcap_{i=1}^{\infty}V_A(S_i)=\emptyset$ and hence $V_A(\bigcup_{i=1}^{\infty}S_i)=\emptyset$. Suppose $S=\bigcup_{i=1}^{\infty}S_i$.
Then, there exists a finite subset $S_0\subseteq S$ such that $S\setminus S_0$ is $A$-stable. Since $S_0$ is finite, so there exists $m$ such
that for all $j\geq m$, the set $V_A(\bigcup_{i=j}^{\infty}S_i)$ does not depend on $j$. Hence for all $j_1, j_2\geq m$, we have
$$
\bigcap_{i=j_1}^{\infty}V_A(S_i)=\bigcap_{i=j_2}^{\infty}V_A(S_i).
$$
This is equivalent to
$$
\bigcup_{i=j_1}^{\infty}C_i=\bigcup_{i=j_2}^{\infty}C_i,
$$
and hence $A^n$ is $\omega$-cocompcat. Now, suppose the language $\mathcal{L}$ is countable and $A$ is equational Artinian. Let $S$ be a system
of equations. We index $S$ by natural numbers as $S=\{ p_i\approx q_i\}_{i=1}^{\infty}$. Suppose $C=A^n\setminus V_A(S)$. We have
$$
C=\bigcup_{i=1}^{\infty}(A^n\setminus V_A(p_i\approx q_i)),
$$
so by our assumption on $\omega$-cocompactness of $A$, there exists $m$ such that for all $j_1, j_2\geq m$,
$$
\bigcup_{i=j_1}^{\infty}(A^n\setminus V_A(p_i\approx q_i))=\bigcup_{i=j_2}^{\infty}(A^n\setminus V_A(p_i\approx q_i)).
$$
Let $S_0=\{ p_1\approx q_1, \ldots, p_{m-1}\approx q_{m-1}\}$. Therefore for any $j$, we have
$$
V_A(S\setminus S_0)=V_A(S\setminus (S_0+p_m\approx q_m+\cdots+p_{m+j}\approx q_{m+j})),
$$
so $S\setminus S_0$ is $A$-stable.
\end{proof}

{\bf Acknowledgement} The authors would like to thank Andreas Blass, Anton Klyachko, Gerhard Paseman and Benjamin Steinberg for their useful comments and suggestions during a discussion in MathOverflow.

\end{document}